\documentclass[12pt,a4paper,reqno]{amsart}

\usepackage{amsmath, amsthm, amssymb, amsrefs,mathrsfs, amsfonts,amssymb,bm}
\usepackage{graphicx}
\usepackage{physics}
\usepackage[bookmarks=true,colorlinks,linkcolor=black]{hyperref}

\theoremstyle{plain}

\newtheorem{theorem}{Theorem}[section]

\newtheorem{corollary}[theorem]{Corollary}

\theoremstyle{definition}

\begin{document}


\author{Mingliang Zhu}
\author{Antal Jo\'os}

\title{Packing $1.35\cdot 10^{11}$ rectangles into a unit square}

\email{lightzhu@csu.edu.cn}
\email{joosa@uniduna.hu}



\maketitle


\begin{abstract}
It is known that $\sum\limits_{i=1}^{\infty} \frac{1}{i (i+1)} = 1$. In 1968, Meir and Moser asked for finding the
smallest $\epsilon$ such that all the rectangles of sizes $1/i \times 1/(i + 1)$ for $i = 1, 2, \ldots$, can be packed into a unit square or a rectangle of area $1 + \epsilon$. In this paper, we show that we can pack the first $1.35\cdot10^{11}$ rectangles into the unit square and give an estimate for $\epsilon$ from this packing.
\end{abstract}

\section{Introduction}

Packing of rectangles means that the rectangles have mutually disjoint interiors.
In 1968, Meir and Moser \cite{meir} asked the following problem: since $\sum\limits_{i=1}^{\infty} \frac{1}{i (i+1)} = 1$, it is reasonable to ask whether the set of
rectangles of sizes  $1/i \times 1/(i + 1)$ for $i =1, 2, \ldots$, can be packed into a square or rectangle of area 1. 
Failing that, find the smallest $\epsilon$ such that the rectangles
can be packed into a rectangle of area $1+\epsilon$.

Meir and Moser proved that $\epsilon \le 0.0678$.
Jennings \cite{Jennings94}, \cite{Jennings95} presented packing such that $\epsilon \le 0.009877$.
Balint \cite{balint}, \cite{balint98}
proved that $\epsilon\le 0.004 004$. Paulhus \cite{paulhus} made a great progress and showed that $\epsilon\le 10^{-9}$. 
Jo\'os \cite{joos-0} pointed out that the proof of the lemma of Paulhus is not true and proved that $\epsilon\le 1.26 \cdot 10^{-9}$ for square container and $\epsilon\le 6.878 \cdot 10^{-10}$ for rectangular shape conatiner. Grzegorek and Januszewski \cite{gj} fixed the lemma of Paulhus.

Meir and Moser asked a similar question in \cite{meir}. Can we pack the squares of side lengths $1/2,1/3,\ldots$ into a rectangle of area $\sum\limits_{i=2}^{\infty}\frac{1}{i^2}=\pi^2/6-1$? 
Failing that, find the smallest $\epsilon$ such that the reciprocal squares
can be packed into a rectangle of area $\pi^2/6-1+\epsilon$.
The answer is not known yet, whether $\epsilon=0$ or $\epsilon>0$, but there are some results. In \cite{meir}, \cite{Jennings94}, \cite{Ball}, \cite{paulhus}, \cite{gj} can be found better and better packings of the squares. Chalcraft \cite{Chalcraft} generalized this question. He packed the
squares of side lengths $1,2^{-t},3^{-t},\ldots$ into a square of
area $\sum\limits_{i=1}^{\infty}i^{-2t}$. He proved that there is a packing of the
squares of side lengths $1,2^{-t},3^{-t},\ldots$ into a square of area $\sum\limits_{i=1}^{\infty}i^{-2t}$ for all $t$ in the range
$[0.5964,0.6]$.\\
W\"astlund \cite{Wastlund} proved if $1/2<t<2/3$, then the squares
of side lengths $1,2^{-t},3^{-t},\ldots$ can be packed into some
finite collection of square boxes of the area $\sum\limits_{i=1}^{\infty}i^{-2t}$.
Jo\'os \cite{joos-2} packed the squares $1,2^{-t},3^{-t},\ldots$ into a rectangle of area $\sum\limits_{i=1}^{\infty}{i^{-2t}}$ for $\log_32<t<2/3$.
Januszewski and Zielonka \cite{jz} extended this interval to $(1/2,2/3]$.
In \cite{joos-1} can be found
a $3$-dimensional generalization of the question of Chalcraft, i.e.
a packing of the $3$-dimensional cubes of edge lengths
$1,2^{-t},3^{-t},\ldots$ into a $3$-box of the right area for all
$t$ in the range $[0.36273,4/11]$. Januszewski and Zielonka \cite{jz} extended this interval to $(2/3,4/11]$. Jo\'os \cite{Joos20} generalized this problem and packed the $d$-cubes of side lengths $1,2^{-t},3^{-t},\ldots$ into a $d$-box of the right area for all $t$ on an interval.
Januszewski and Zielonka \cite{jz2} extended this interval to $(1/d,2^{d-1}/(d2^{d-1}-1)]$.

Tao \cite{tao} recently proved that any $1/2<t<1$, and any $i_0$ that is sufficiently large depending on $t$, the squares of side length $i^{-t}$ for $i\ge i_0$ can be packed into a square of area $\sum\limits_{i=i_0}^{\infty}{i^{-2t}}$.
Sono \cite{Sono} generalized the result of Tao and considered the squares of side length $f(i)^{-t}$. He proved that for any $1/2<t<1$, there exists a positive integer $i_0$ depending on $t$ such that for any $i\ge i_0$, squares of side length $f(i)^{-t}$ for $i\ge i_0$ can
be packed into a square of area $\sum\limits_{i=i_0}^{\infty}{f(i)^{-2t}}$ if the function $f$
satisfies some suitable conditions.
McClenagan \cite{McClenagan} proved that if $\frac{1}{d}<t<\frac{1}{d-1}$, and $i_0$ is sufficiently large depending on $t$, then the $d$-cubes of
side length $i^{-t}$ for $i\ge i_0$ can be perfectly packed into a $d$-cube of volume $\sum\limits_{i=i_0}^{\infty}{f(i)^{-dt}}$.

In the paper of Tao \cite{tao} one can read the following for any $1/2<t<1$. "We remark that the same argument (with minor notational changes) would also allow one to pack rectangles of dimensions $n^{-t}\times (n + 1)^{-t}$ for $n\ge n_0$ perfectly into a square of area $\sum\limits_{n=n_0}^{\infty}\frac{1}{n^{t}(n+1)^{t}}$; we leave the details of this modification to the interested reader."\\
The case $t=1$ is unsolved for packing squares of side lengths $1/2,1/3,\ldots$ into a square of rectangle of area
$\sum\limits_{i=2}^{\infty}{1/i^2}=\pi^2/6-1$ or for packing rectangles of dimensions $1\times1/2,1/2\times1/3,\ldots$ into a square of rectangle of area $\sum\limits_{i=1}^{\infty}{1/(i(i+1))}=1$.
We consider this last question.

\section{Result}

We follow the algorithm in \cite{joos-0} and write a Julia program to run out the results for different number of packed rectangles.
Our packing result for 1000 rectangles is in Fig. \ref{fig:1}

\begin{figure}[htbp]
\centering
\includegraphics[width=0.8\textwidth]{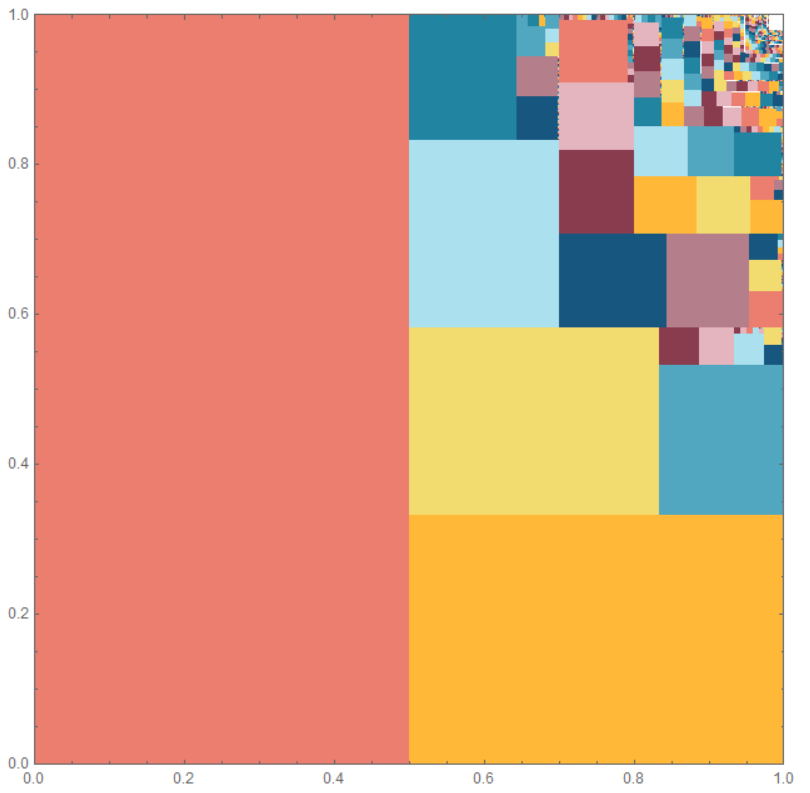}
\caption{Packing 1000 rectangles }
\label{fig:1}
\end{figure}

\subsection{Result from computer program}

We call rectangles that we pack and we call (empty) boxes which are the remaining (rectangular shape) empty spaces to avoid confusion.
By using computer program, we have packed the first $10^{11}$ rectangles into the unit square. The left largest empty box (let $E$ be this box) has a width and length as
$1.888\ 883\ 876\ 3176\ 668\cdot 10^{-6}\times 1.888\ 893\ 876\ 343\ 809\ 9\cdot 10^{-6}$.
From the sum of the area sequence to $10^{11}$, we can estimate the remaining area to be packed is $\frac{1}{10^{11} + 1} <  10^{-11}$. Observe, the side lengths of $E$ are less than $\sqrt{10^{-11}}= \sqrt{10} \cdot 10^{-6}$.

We are going to pack rectangles from $10^{11} + 1$ and on, into $E$. Since $E$ is close to a square, there are at least $188\ 888^2>3.5\cdot 10^{10}$ following rectangles can be packed into $E$. So we can pack at least $1.35 \cdot 10^{11}$ into the unit square.

\subsection{Ratio of largest rectangle to the total remaining area }

When we calculate the ratio of largest empty box to the total remaining area, we find a very interesting phenomenon: the ratio is near to $0.36$ when number of packed rectangles $n$ becomes larger. The values are listed in Table. \ref{table:1}.

\begin{table}[htb]
\begin{center}
\caption{Ratio table}
\label{table:1}
\begin{tabular}{|c|c|}
\hline   \textbf{number of rectangles} & \textbf{ratio} \\
\hline $ 10^3  $           & 0.4142 \\
\hline $ 10^4   $          & 0.3441 \\
\hline $ 10^5   $          & 0.3577 \\
\hline $ 10^6   $          & 0.3554 \\
\hline $ 10^7   $          & 0.3502 \\
\hline $ 10^8   $          & 0.3400 \\
\hline $ 10^9   $          & 0.3701 \\
\hline $ 2 \cdot 10^9  $  & 0.3648 \\
\hline $ 4 \cdot 10^9  $  & 0.3580 \\
\hline $ 5 \cdot 10^9  $  & 0.3613 \\
\hline $ 1 \cdot 10^{10}$ & 0.3687 \\
\hline $ 2 \cdot 10^{10}$ & 0.3677 \\
\hline $ 5 \cdot 10^{10}$ & 0.3631 \\
\hline $ 1 \cdot 10^{11}$ & 0.3568 \\
\hline
\end{tabular}
\end{center}
\end{table}

During the packing of the first $10^5$ rectangles into the unit square the ratios of largest empty box to the total remaining area can be seen in Fig. \ref{fig:1_5}.

\begin{figure}[htbp]
\centering
\includegraphics[width=0.8\textwidth]{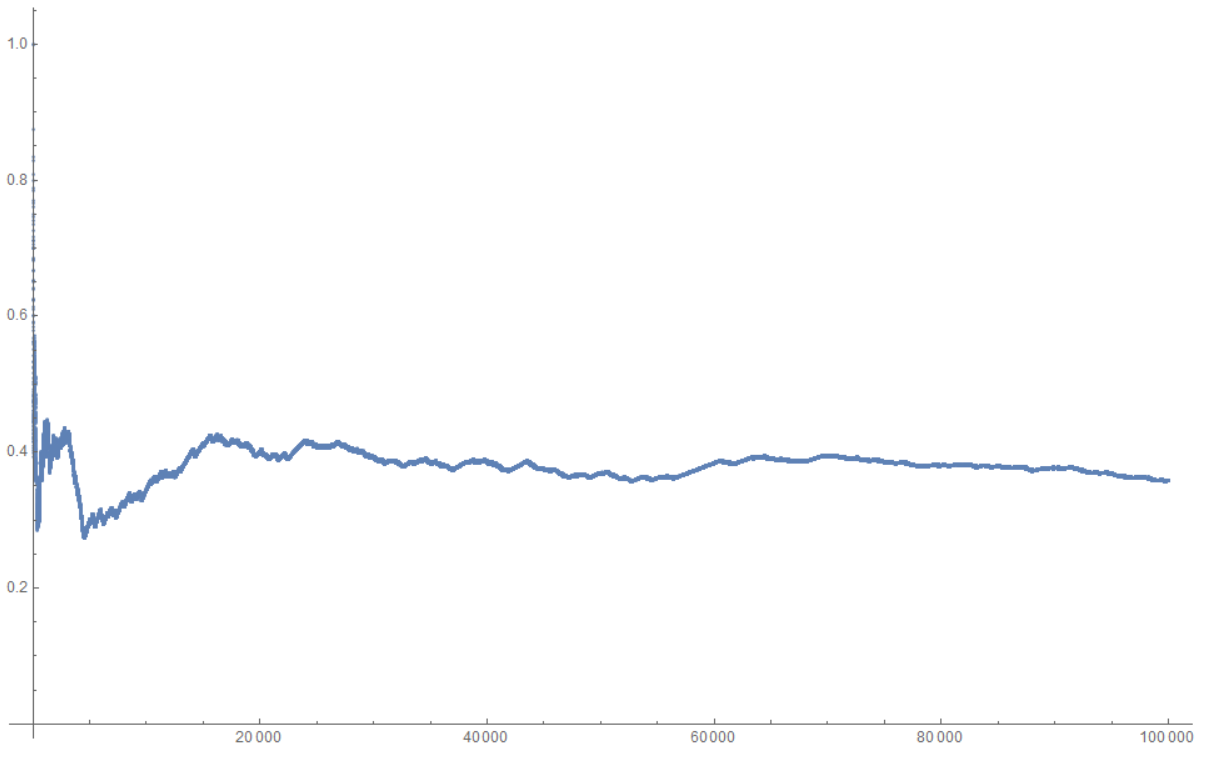}
\caption{The ratios of largest empty box to the total remaining area}
\label{fig:1_5}
\end{figure}

\subsection{Theorem for estimating the area}

It is easy to prove the following theorem, though the actual area is greater than the Balint's rectangle area $1 + 6 / (5n)$ in \cite{balint98}. We can use a small trick to decrease the total area of $1+\epsilon$.

\begin{theorem}
\label{}
If $n - 1$ ($n\ge1000$) rectangles can be packed into the unit square, then all rectangles $\sum\limits_{i=1}^{\infty} \frac{1}{i (i+1)}$ can be packed into the square of side length $1 + 1/n$. The actual packing area is less than $1 + 2/n \times (\ln 2 + 1 / 2n)$
\end{theorem}

\begin{proof}[Proof]

Denote the rectangle of dimensions $\frac{1}{n} \times \frac{1}{n+1}$ by $P_n$.
Assume $P_1,\ldots,P_{n-1}$ are packed into the unit square.\\
Observe, $P_n$ can be packed into the square of side length $\frac{1}{n}$.
Pack rectangles from $n$ to infinity in this way of rows and columns into a large rectangle as Fig. \ref{fig:2} showing.
The rectangles from $P_{2^{i-1}n}$ to $P_{2^{i}n-1}$ lie in the $i$-th row for $i=1,2,\ldots$. The width of the smallest rectangular shape container is
$$\frac{1}{n} + \frac{1}{2n} + \frac{1}{4n} + \cdots  = \frac{2}{n}$$
and the length of the smallest rectangular shape container is
$$\frac{1}{n} + \frac{1}{n+1} + \frac{1}{n+2} + \cdots + \frac{1}{2n-1}.$$
\begin{figure}[htbp]
\centering
\includegraphics[width=0.8\textwidth]{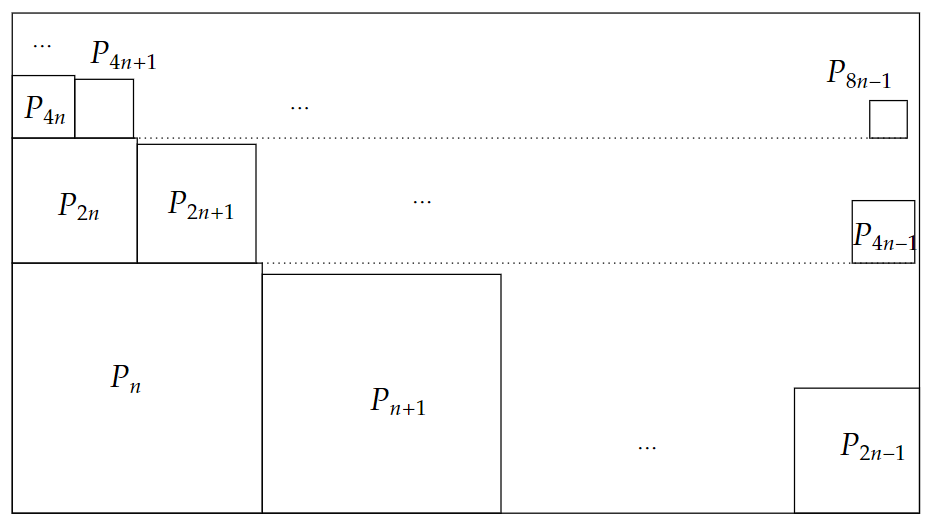}
\caption{packing 1000 rectangles }
\label{fig:2}
\end{figure}

Since $\frac{1}{n+1} + \frac{1}{n+2} + \cdots + \frac{1}{2n-1} + \frac{1}{2n}<\int_{n}^{2n}\frac{1}{x}\ dx=\ln 2$,
the length of the container:
$$
\frac{1}{n} + \frac{1}{n+1} + \frac{1}{n+2} + \cdots + \frac{1}{2n-1} + \frac{1}{2n} - \frac{1}{2n}$$
$$= \frac{1}{n} +\left( \frac{1}{n+1} + \frac{1}{n+2} + \cdots + \frac{1}{2n-1} + \frac{1}{2n}\right) - \frac{1}{2n}$$
$$< \frac{1}{n} + \ln2 - \frac{1}{2n} = \ln 2 +  \frac{1}{2n}.
$$
The length of the $i$-th row is $\ln 2+\frac{1}{2^{i-1}n}<\ln 2+\frac{1}{2n}$ for $i=1,2,\ldots$.
Since $n\ge 1000$, this length $\ln 2 + \frac{1}{2n}$ is less than $1$.

The rectangles from $n$ to infinity are put into a box of $\frac{2}{n} \times (\ln2 + \frac{1}{2n})$. Divide this box into two equal smaller boxes of dimensions $\frac{1}{n} \times (\ln2 + \frac{1}{2n})$ and glue to the unit square as in Fig.\ref{fig:3}.
\begin{figure}[htbp]
\centering
\includegraphics[width=0.5\textwidth]{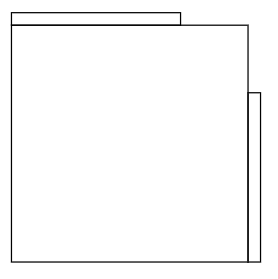}
\caption{final packing }
\label{fig:3}
\end{figure}

We proved that all rectangles can be packed into a square of side length $(1+1/n)$, and it is just a rough estimate, although in two stripes it is not efficient to pack them, and can be improved by some ways obviously like \cite{balint}.
\end{proof}

From the computer program result of subsection 2.1, we have

\begin{corollary}
  The rectangles of dimensions $1/n \times 1/(n + 1)$ for $n \geq 1$ can be packed into a square of side length $1 + \frac{1}{1.35 \cdot 10^{11} + 1 }$, which shows that $\epsilon <1.49\cdot 10^{-11}$.
\end{corollary}

\section{Conclusion}

We hope to pack more rectangles into the unit square by computer program, though the computer memory is the constraint of improvement. A mathematical proof for this problem might be needed.

\end{document}